\documentclass{amsart}
\usepackage{graphicx}
\usepackage{amssymb}
\usepackage{epstopdf}
\usepackage{amsmath}
\usepackage{tikz}
\usepackage{subfigure}
\usepackage{multirow}

\newtheorem{theorem}{Theorem}[section]

\newtheorem{proposition}{Proposition}[section]

\newcommand{\R}{\mathbb{R}}
\newcommand{\proj}{\operatornamewithlimits{proj}}

\title{Optimal Radiotherapy Treatment Planning using Minimum Entropy models}
\author{Richard Barnard$^{\ast}$ ,
 Martin Frank, Michael Herty
 \\\vspace{6pt} {\em{RWTH Aachen University, Aachen Germany}}
}
\thanks{$^\ast$Corresponding author. Email: barnard@mathcces.rwth-aachen.de}
\begin{document}
\maketitle
\begin{abstract}
We study the problem of finding an optimal radiotherapy treatment plan. A time-dependent Boltzmann particle transport model is used to model the interaction between radiative particles with tissue.  This model allows for the modeling of inhomogeneities in the body and allows for anisotropic sources modeling distributed radiation---as in brachytherapy---and external beam sources---as in teletherapy.  We study two optimization problems: minimizing the deviation from a spatially-dependent prescribed dose through a quadratic tracking functional; and minimizing the survival of tumor cells through the use of the linear-quadratic model of radiobiological cell response.  For each problem, we derive the optimality systems.  In order to solve the state and adjoint equations, we use the minimum entropy approximation; the advantages of this method are discussed.  Numerical results are then presented.
\end{abstract}\bigskip
\section{Introduction}

Radiotherapy  is one of the main tools currently in use for the treatment of cancer.  Radiation is deposited in the tissue with the aim of damaging tumor cells and disrupting their ability to reproduce.  In this paper, we are interested in the treatment planning problem. We wish to determine a method of delivering a sufficient level of radiative energy to the tumor cells to ensure cell death.  This is balanced by our our desire to minimize damage to healthy tissue, especially specific crititcal structures---which we call regions at risk---that should be damaged as little as possible.  This problem can be divided into two phases.  First, one must determine an appropriate external source of radiation that optimally balances these two competing goals.  Secondly, one must devise a means to administer this external source through (in the case of teletherapy) the selection of beam angles and similar parameters.  We focus on the first of these steps.  Once the optimal source is determined, the means of creating such a source would still need to be obtained.  Often, the optimal dose will be determined manually and iteratively by an experienced dosimetrist;  this dose consists of selecting several fixed beam angles and strengths.  However, newer methods such as Intensity Modulated Radiation Therapy allow for dynamic sources; such a method requires more automated algorithms for determining the treatment scheme \cite{SheFerOli99}.  

Quantifying what the ``best'' treatment plan is difficult and still unsettled (see \cite{SheFerOli99},\cite{KufMonSch09} for instance).  Often a heuristic approach is used by technicians.  Here, we will study two different means of measuring the effectiveness of a given source.  First, we use a general measure: we look to minimize the deviation from a desired distribution of radiation, similar to the objective functional used in \cite{SheFerOli99}, \cite{Bor99},\cite{FraHerSch08}, and \cite{FraHerSan10}.  The other method we focus on involves using radiobiological parameters describing cells' susceptibility  to damage from radiation.  The cells' response to the delivered dose is modeled by the frequently-used linear-quadratic model for cell death.   A significant difficulty in using this model is that the parameters for cell damage are dependent on several factors leading to significant levels of interpatient heterogeneity \cite{SteDowPea86},\cite{Web94},\cite{FerMavAda08},\cite{SouParEva08}.  Furthermore, the parameters can change over the course of treatment plans as patients exhibit changing levels of radiosensitivity.  In particular, \cite{FerMavAda08} studied the optimal treatment planning problem and the dependence of the optimal dose on different parameters, especially in the face of changing radiosensitivity.  In \cite{SouParEva08}, a method for treatment planning that accounts for the uncertainty in an individual's specific linear-quadratic parameters was developed.  

As described in \cite{FraHerSch08}, the Fermi-Eyges theory of radiation is used for most dose calculation algorithms; this theory fails to incorporate properly physical interactions in areas where inhomogeneities such as void-like regions occur.  Here, we will make use of a Boltzmann transport model which describes the scattering and absorbing effects of the tissue while also allowing for inhomogeneities.  Our optimization method here differs in two significant ways from that in \cite{FraHerSch08}.  First, our model is time-dependent; we will allow for time-dependent, anisotropic sources which are external (teletherapy) as well as sources distributed in the tissue (brachytherapy).  Second, as opposed to the spherical harmonics approximation used there, we will make use of a minimum entropy approximation to the Boltzmann transport model which has the benefit of maintaining several desirable physical properties from the original model.  

In Section \ref{sec:model}, we will introduce the Boltzmann transport model describing the interactions of the radiation with the tissue and the two objective functionals which we use to measure the effectiveness of a given dose.  In Section \ref{sec:optcond} we establish optimality conditions for both objective functionals through the adjoint equations of the Boltzmann equation.  From there, we introduce the minimum entropy ($M_1$) approximation and briefly discuss the advantages of this approximation.  Numerical results in two dimensions are presented for several test cases in Section \ref{sec:numtest}.   

\section{Mathematical Model}
\label{sec:model}
\subsection{Transport Equation}
\label{sec:transeq}
Let $Z\subset\R^3$ be an open, convex, bounded domain with smooth boundary and outward normal $n(x),$ and $T>0$ be given.  The direction of a particle's movement is $\Omega\in S^2$ where $S^2$ is the 3-dimensional unit sphere.  Then we model particle transport with sources $q,q_1:[0,T]\times\R^3\times S^2\rightarrow\R$ by the following equation
\begin{align}
\label{eq:transport}
\psi_t(t,x,\Omega)+\Omega \cdot \nabla_x\psi(t,x,\Omega)+\sigma_t(x)\psi(t,x,\Omega)&=\sigma_s(x)\int_{S^2} s(x,\Omega\cdot\Omega')\psi(t,x,\Omega')d\Omega'\\
&+q(t,x,\Omega)\notag
\end{align}
with boundary condition
\begin{equation}
\label{eq:tempbdry}
\psi(t,x,\Omega)=q_1(t,x,\Omega),~\forall (t,x,\Omega)\in\Gamma^-:=\{(t,x,\Omega):t\in[0,T],x\in\partial Z,n(x)\cdot\Omega<0\}
\end{equation}
and initial condition
\begin{equation}
\label{eq:initcond}
\psi(0,x,\Omega)=0,~\forall (x,\Omega)\in Z\times S^2.
\end{equation}

It should be noted that, (\ref{eq:transport}) does not explicitly include the energy dependence of the transport of particles.  However, as detailed in \cite{FraHerSan10}, the Boltzmann Continuous Slowing Down model, which does include energy dependence, can be transformed into a problem such as (\ref{eq:transport}) where the initial-value condition intuitively describes the physically reasonable condition that no particles are of ``infinite'' energy.  In general, the term $q$ denotes a distributed radiative source which would be appropriate for applications in brachytherapy; meanwhile, $q_1$ seems a natural means of modeling the effects of external beams in teletherapy applications.  However, as described in Section \ref{sec:M1}, the role of the control on the incoming boundary is difficult to model in the context of the $M_1$ approximation (along with similar moment approximations) to (\ref{eq:transport}).  With this in mind, throughout this paper we will assume $q_1\equiv0;$ that is, we will replace (\ref{eq:tempbdry}) with the incoming boundary condition
\begin{equation}
\label{eq:bdrycond}
\psi(t,x,\Omega)=0,~\forall(t,x,\Omega)\in\Gamma^-.
\end{equation}
We will address in Section\ref{sec:M1} a means of modeling the case of teletherapy using the distributed control.
The quantity $\psi(t,x,\Omega)\cos\theta dAd\Omega$ is interpreted as the number of particles at time $t$ passing through $dA$ at $x$ in the direction $d\Omega$ near $\Omega$ where $\theta$ is the angle between $\Omega$ and $dA.$ The total cross-section $\sigma_t(x)$ is the sum of the absorption cross-section $\sigma_a(x)$ and scattering cross-section $\sigma_s(x).$  Finally, the scattering kernel---intuitively describing the probability of a particle changing direction from $\Omega'$ to $\Omega$ at $x\in Z$ after a scattering event---is normalized; that is, it satisfies for all $x,$

\begin{equation*}
\int_{-1}^{1}s(x,\eta)d\eta=\frac{1}{2\pi}.
\end{equation*}
In this paper, we will usethe simplified Henyey-Greenstein kernel (see \cite{AydOliGod02} for more)
\begin{equation*}
s_{HG}(x,\eta):=\frac{1-g(x)^2}{4\pi(1+g(x)^2-2g(x)\eta)^{3/2} }
\end{equation*}
where $g$ is the average cosine of the scattering angle.
Following \cite{FraHerSch08},  we define the operators
\begin{align*}
({\bf A}\psi)(t,x,\Omega)&=-\psi_t(t,x,\Omega)-\Omega\cdot \nabla_x\psi(t,x,\Omega)\\
({\bf K}\psi)(t,x,\Omega)&=\sigma_s(x)\int_{S^2}s(x,\Omega' \cdot \Omega)\psi(t.x,\Omega')d\Omega'\\
(\Sigma\psi)(t,x,\Omega)&=\sigma_t(x)\psi(t,x,\Omega)\\
\end{align*}

and make the following general assumptions:
\begin{enumerate}
\renewcommand{\theenumi}{(H-\arabic{enumi})}
\renewcommand{\labelenumi}{(H-\arabic{enumi})}
\item \label{hyp:positivity}
$\sigma_t,\sigma_s,s\geq0,$
\item \label{hyp:bounded} $\sigma_t,\sigma_s\in L^\infty(Z),$
\item \label{hyp:absorb} $\sigma_t(x)-\sigma_s(x)\geq\alpha>0,~\forall x\in Z.$
\end{enumerate}

We note that \ref{hyp:positivity} and \ref{hyp:bounded} are reasonable physically and that as long as the medium under consideration is absorbing, the coercivity hypothesis \ref{hyp:absorb} is satisfied. These assumptions imply that the operators
\begin{equation*}
{\bf K},\Sigma:L^2([0,T]\times Z\times S^2)\rightarrow L^2([0,T]\times Z\times S^2)
\end{equation*}
are bounded and linear operators.  Finally, we note that if we define
\begin{equation*}
 D({\bf A})=\{\psi\in L^2([0,T]\times Z\times S^2):\psi_t,\Omega\cdot \nabla_x\psi\in L^2([0,T]\times Z\times S^2)\},
\end{equation*}
then
\begin{equation*}
{\bf A}:D({\bf A})\rightarrow L^2([0,T]\times Z\times S^2)
\end{equation*}
is a well-defined linear operator.  We also let $\widetilde{D(A)}$ denote the set of functions of $D(A)$ which also satisfy (\ref{eq:bdrycond}) and (\ref{eq:initcond}), for $q_1\equiv 0.$  Then we can conclude---for instance, by Ch. XXI, 2, Theorem 3 of \cite{DauLio93} the following holds.
\begin{theorem}
Suppose that \ref{hyp:positivity}-\ref{hyp:absorb} hold and that $q_1\in L^2([0,T]\times Z\times S^2)$ is nonnegative.  Then (\ref{eq:transport}) has a unique solution $\psi\subset\tilde{D(A)}.$
\end{theorem}
We define the control-to-state mapping
\begin{equation*}
\mathcal{E}:L^2([0,T]\times Z\times S^2)\rightarrow D({\bf A})
\end{equation*}
as the map which maps a source $q$ to the corresponding solution $\psi$ to 
\begin{equation*}
-{\bf A}\psi+\Sigma\psi={\bf K}\psi+q
\end{equation*}
subject to (\ref{eq:bdrycond}) and (\ref{eq:initcond}). By the above, this is a well-defined, linear operator.  As in \cite{FraHerSch08}, it is also bounded.

\begin{proposition}
Let $\langle \cdot,\cdot\rangle_2$ denote the standard inner product on $L^2([0,T]\times Z\times S^2)$.  Then it holds that, for $\psi$ satisfying (\ref{eq:bdrycond}) and (\ref{eq:initcond}),
\begin{equation*}
\langle-{\bf A}\psi+\Sigma\psi-{\bf K}\psi,\psi\rangle_2\geq\alpha||\psi||^2_2
\end{equation*}
and thus $\mathcal{E}$ is a bounded operator.
\begin{proof}
We note that, by an argument analogous to the proof of Lemma 3.1 of \cite{FraHerSch08}, 
\begin{equation*}
\langle \Sigma \psi-{\bf K}\psi,\psi\rangle_2\geq\alpha||\psi||^2_2.
\end{equation*}
Additionally, by integration by parts and the initial and boundary conditions,
\begin{align*}
\langle -{\bf A}\psi,\psi\rangle_2&=\int_{[0,T]\times Z\times S^2}(\psi_t\psi) dtdxd\Omega+\int_{[0,T]\times Z\times S^2}(\Omega\cdot\nabla_x\psi)\psi dtdxd\Omega\\
&=\int_{Z\times S^2}\frac{1}{2}\psi^2(T,x,\Omega)dxd\Omega+\int_{[0,T]}\int_Z \text{div}\Big(\int_{S^2}\frac{1}{2}\Omega\psi^2d\Omega\Big)dxdt\\
&=\int_{Z\times S^2}\frac{1}{2}\psi^2(T,x,\Omega)dxd\Omega+\int_{[0,T]}\int_{\partial Z\times S^2}\frac{1}{2}(\Omega\cdot n(x))^+\psi^2dxd\Omega dt\geq 0.
\end{align*}
Combining the two inequalities gives the desired result.
\end{proof}
\end{proposition}
\subsection{Objective Functionals}
\label{sec:costs}
We consider two objective functionals.  In each case, we neglect the incoming boundary source term by fixing $q_1\equiv 0$ for reasons which are discussed in \ref{sec:M1}.  Also, we are primarily interested only in the total dose deposited as a function of $x$, which we denote by the operator $D:L^2([0,T]\times Z\times S^2)\rightarrow L^2(Z)$ which is defined as
\begin{equation*}
D\psi(x):=\int_{[0,T]\times S^2}\psi(t,x,\Omega)dt d\Omega.
\end{equation*}
The first functional is a quadratic tracking which measures the deviation from a prescribed dose $\overline{D}(x)$
\begin{equation}
\label{eq:tracking}
J_T(\psi,q):=\int_Z c_1(D\psi-\overline{D})^2dx+\int_Z c_2 D(q)^2dx.
\end{equation}
Here $c_1,c_2>0$ are spatially dependent weighting functions.  One such weighting function, used in \cite{FraHerSch08}, for $c_1$ is
\begin{equation*}
c_1=c_T\chi_{Z_T}+c_R\chi_{Z_R}+c_N\chi_{Z_N},
\end{equation*}
where $c_T,c_R,c_N$ are constants and $Z_T,Z_R,Z_N$ are, respectively, the portions of $Z$ corresponding to the tumor, a region at risk, and normal tissue with $Z=Z_T\cup Z_R\cup Z_N.$  This functional differs slightly from that studied in \cite{FraHerSch08} in that we are concerned with the dose deposited but the source is allowed to be angle-dependent.  In that work, the objective functional studied for anisotropic controls also assumed a desired anisotropic distribution was to be tracked, as opposed to simply a total dose.  Though this objective functional does not attempt to model cell response to the dose, it does have the benefits of not requiring as inputs the radiobiological parameters needed for such a model. 

The second objective functional involves modeling cell death as a function of $D(\psi)$.  We will use the linear-quadratic model (\cite{SouParEva08}), neglecting cell growth and repair rates.  In a reference volume for a given dose $D$, the fraction of surviving cells of a single type is given by
\begin{equation}
\label{eq:survfrac}
SF=\exp(-\alpha D-\beta D^2)
\end{equation}
where $\alpha,\beta$ are parameters dependent the type of cell being irradiated.  This model has been used extensively for determining dose strategies (for instance, \cite{Dal85, SteDowPea86, SteDeaDuc87, Web94, KinDiPWaz00} and, in a modified form, \cite{FerMavAda08}). The surviving fraction is often used to determine the tumor control probability for a given region:
\begin{equation*}
TCP=\exp\int\Big(-\rho\exp(-\alpha_iD_i-\beta_i D^2_i)\Big)
\end{equation*}
where $\rho$ is the density of tumor cells as a function of space.  We consider a continuous measure of the number of surviving cells of the various cell types in $Z.$  We consider a single type of tumor cell and assign this cell type the index $0$.  The remaining cell types are given index $i=1,\dots, N.$  Then, for cell densities $\rho_i(x)\geq0$ with $\sum_{i=0}^N \rho_i(x)=1,$ we define
\begin{align}
J_{SF}(\psi,q)&:=a_0\int_Z \rho_0\exp\big[-\alpha_0 D\psi-\beta_0 (D\psi)^2\big]\label{eq:survfraccost}\\
&+\sum_{i=1}^N a_i\int_Z\rho_i\Big(1-\exp\big[-\alpha_iD\psi-\beta_i (D\psi)^2\big] \Big)\notag\\
&+\frac{c_2}{2}\int_{Z}(Dq)^2\notag
\end{align}
where we assign constant weights for each cell type $a_i>0$ and the control $c_2>0.$  

For either cost functional, we set $U(x)\geq 0$ to be the maximum allowed source at $x\in Z.$  Let $L^2_{ad}([0,T]\times Z\times S^2)$ denote the set of admissible controls--that is, those controls satisfying $q\geq 0$ and 
\begin{equation*}
\int_{S^2} q(t,x,\Omega)d\Omega \leq U(x)
\end{equation*}
almost everywhere.  Note that this subset of $L^2( [0,T]\times Z\times S^2)$ is nonempty, closed, and convex.

\section{Optimality conditions}
\label{sec:optcond}
We are, then, interested in two separate optimal control problems: the minimization of tracking error
\begin{equation*}
\tag{$\mathcal{P}_T$}  \min_{\psi\in \widetilde{D(A)}, q\in L^2_{ad}[0,T]\times Z\times S^2)} J_T(\psi,q),~
s.t.~\mathcal{E}(q)=\psi.
\end{equation*}
and the minimization of tumor cell surviving fraction balanced against normal cell death
\begin{equation*}
\tag{$\mathcal{P}_{SF}$}  \min_{\psi\in \widetilde{D(A)}, q\in L^2_{ad}([0,T]\times Z\times S^2)} J_{SF}(\psi,q),~
s.t.~\mathcal{E}(q)=\psi.
\end{equation*}
In order to obtain optimality conditions for $\mathcal{P}_T$ and $\mathcal{P}_{SF},$ we follow the example of \cite{Tro10}.  One of the advantages of using $J_T$ as the objective functional is the following:
\begin{theorem}
Given \ref{hyp:positivity}-\ref{hyp:absorb}, and $\overline{\psi}\in\widetilde{D(A)}$ with $\overline{D}=D(\overline{\psi})$ and $c_1,c_2>0$, then $\mathcal{P}_T$ has a unique solution $q^*\in L^2_{ad}([0,T]\times Z\times S^2)$.
\end{theorem}
\begin{proof}
We note that the reduced objective functional $j_T(q)=J_T(\mathcal{E}(q),q)$ is a convex functional over the closed,convex set $L^2_{ad}([0,T]\times Z\times S^2)$ of the Hilbert space $L^2([0,T]\times Z\times S^2)$ and that $\mathcal{E}$ is linear and bounded.  By Theorem 2.16 of \cite{Tro10}, there exists a unique minimizer for problem $\mathcal{P}_T$. 
\end{proof}
We now turn to the existence of an optimal control for $\mathcal{P}_{SF}.$  Because the associated reduced objective functional is nonconvex, we can not assume existence of a unique optimal control.
\begin{theorem}
Given \ref{hyp:positivity}-\ref{hyp:absorb} and $c_i,a>0$ there exists at least one optimal control $q^*\in L^2_{ad}([0,T]\times Z\times S^2)$ for $\mathcal{P}_{SF}.$
\end{theorem}
\begin{proof}
We note that $J_{SF}$ is convex with respect to $q$ and that the set of admissible controls is closed, convex, and bounded.  This, coupled with the boundedness of the operator $\mathcal{E}$ allows us to conclude that the infimum of $J_{SF}$ over $ L^2_{ad}([0,T]\times Z\times S^2)$  is finite and that any weakly convergent minimizing sequence of controls $q_n$ will converge to an admissible control $q^*$.  By the linearity of ${\bf A},\Sigma,{\bf K},$ we conclude that the corresponding solutions converge strongly to a solution corresponding to the $q^*.$  (see Theorem 5.7 of \cite{Tro10}).  The convexity of $J_{SF}$ as a functional of $q$ allows us to conclude this control is optimal, as in Theorem 4.15 of \cite{Tro10}.
\end{proof}

We next note the following from \cite{DauLio93}:
\begin{proposition}
\label{prop:adjoint}
Under \ref{hyp:positivity}-\ref{hyp:absorb}, for any $r\in L^2([0,T]\times Z\times S^2),$
\begin{align*}
\bf{A}\lambda +\Sigma\lambda -{\bf K}\lambda&=r \\
\lambda&=0 \text{ on } \Gamma^+\\
\lambda(T,x,\Omega)&=0,~\forall (x,\Omega)\in Z\times S^2.
\end{align*}
has a unique solution $\lambda\in L^2([0,T]\times Z\times S^2),$ and the solution operator $\mathcal{E^*}$ is linear, bounded, and the adjoint to $\mathcal{E}.$
\end{proposition}
Now, we are ready to turn to the optimality conditions for $\mathcal{P}_T,$ which is similar to that found in \cite{FraHerSch08}.

\begin{theorem}
\label{thm:trackopt}
For a desired distribution $\overline{D}\in L^2(Z)$ and under the assumptions \ref{hyp:positivity}-\ref{hyp:absorb}, $q^*\in  L^2_{ad}([0,T]\times Z\times S^2)$ is the unique optimal solution to $\mathcal{P}_T$ if and only if 
\begin{equation}
q^*(t,x,\Omega)=\proj_{[0,U(x)]}\Big(-\frac{1}{c_2(x)}\lambda^*(t,x,\Omega)\Big)
\end{equation}
where $\lambda^*=\mathcal{E}^*(c_1(D(\psi^*)-\overline{D})),$ $\psi^*=\mathcal{E}(q^*)$, and $\proj_S(\cdot)$ denotes projection onto the closed set $S.$
\end{theorem}
\begin{proof}
The gradient of $j_T$ is $c_1\mathcal{E}^*(D\mathcal{E}(q^*)-\overline{D}))+c_2Dq$ and so we conclude, using the same arguments of Theorem 2.25, Lemma 2.26, and Theorem 2.27 of \cite{Tro10}. 
\end{proof}
This means that the optimality system for $\mathcal{P}_T$ is
\begin{equation}
\label{eq:trackoptsys}
\begin{cases}
-{\bf A} \psi^*+\Sigma \psi^*&={\bf K}\psi^*+q^*,\\
{\bf A} \lambda^*+\Sigma \lambda^*&={\bf K}\lambda^*+c_1D(\mathcal{E}(q^*)-\overline{D}),\\
q^*&=\proj_{[0,U(x)]}(-\frac{1}{c_2}\lambda^*)
\end{cases}
\end{equation}
with $\psi^*$ satisfying (\ref{eq:bdrycond}) and (\ref{eq:initcond}),  $\lambda^*=0$ on $\Gamma^+,$ and $\lambda(T,x,\Omega)=0$ on $Z\times S^2.$

The necessary condition for the optimality of a source for $\mathcal{P}_{SF}$ is of a similar form. 

\begin{theorem}
\label{thm:sfopt}
For a given distribution of cells $\rho_i,~i=0,1,\dots N$ and the associated linear-quadratic cell death parameters $\alpha_i,\beta_i>0,$ a locally optimal $q^*\in  L^2_{ad}([0,T]\times Z\times S^2)$ for $J_{SF}$ must be such that
\begin{equation}
\label{eq:sfproj}
q^*(t,x,\Omega)=\proj_{[0,U(x)]}\Big(-\frac{1}{c_2(x)}\lambda^*(t,x,\Omega)\Big)
\end{equation}
where $\lambda^*=\mathcal{E}^*(r^*)$ for
\begin{align*}
r^*=&\Big[\big(-\alpha_0-2\beta_0D\psi^*\big)\big(c_0\rho_0\exp[-\alpha_0 D\psi^*-\beta_0(D\psi^*)^2]\big)\Big]\\
&-\sum_{i=1}^N\Big[\big(-\alpha_i-2\beta_i D\psi^*\big)\big(c_i\rho_i\exp[-\alpha_i D\psi^*-\beta_i(D\psi^*)^2]\big)\Big]
\end{align*}
and $\psi^*=\mathcal{E}(q^*).$
\end{theorem}
\begin{proof}
The reduced cost functional $j_{SF}(q)=J_{SF}(\mathcal{E}(q),q)$ is Fr\'echet differentiable due to the differentiability of the integrands in $J_{SF}$ and the linearity and continuity of $\mathcal{E}.$  For simplicity, we assume $\rho_0\equiv1$ (and thus $\rho_i\equiv0$ for $i\geq1$) first.  We then get that for a locally optimal $q^*$ and any other admissible control $q,$
\begin{equation*}
j_{SF}'(q^*)(q-q^*)\geq0
\end{equation*}
 where, by the linearity of the adjoint solution operator $\mathcal{E}^*$ and the chain rule,
\begin{align*}
j_{SF}'(q^*)(q-q^*)&= c_0\int_Z\big(-\alpha_0-2\beta_0D\mathcal{E}(q^*)\big)\big(c_0\rho_0\exp[-\alpha_0 D\mathcal{E}(q^*)-\beta_0(D\mathcal{E}(q^*))^2]\big)\mathcal{E}(q-q^*)\\
&+c_2\int_Z q^*(q-q^*)\\
&=\int_Z\mathcal{E}^*((-\alpha_0-2\beta_0D\mathcal{E}(q^*)\big)\big(c_0\rho_0\exp[-\alpha_0 D\mathcal{E}(q^*)-\beta_0(D\mathcal{E}(q^*))^2]\big)(q-q^*)\\
&+c_2\int_Zq^*(q-q^*).
\end{align*}
This leads to the minimum principle (by an argument analogous to that found on pages 68-70 of  \cite{Tro10}), which states that
\begin{equation*}
\min_{q\in L^2_{ad}([0,T]\times Z\times S^2)}[(\lambda^*+c_2q^*)q]
\end{equation*}
is attained almost everywhere at $q^*$ which leads us to conclude that $q^*$ satisfies (\ref{eq:sfproj}). Now if we drop the assumption that $\rho_i\equiv1,$ and allow for the presence of other species, the argument follows through in a nearly identical fashion.
\end{proof}
This gives the necessary optimality system
\begin{equation}
\label{eq:sfoptsys}
\begin{cases}
-{\bf A} \psi^*+\Sigma \psi^*&={\bf K}\psi^*+q^*,\\
{\bf A} \lambda^*+\Sigma \lambda^*&={\bf K}\lambda^*+r^*,\\
q^*&=\proj_{[0,U(x)]}(-\frac{1}{c_2}\lambda^*)
\end{cases}
\end{equation}
where $r^*$ is as in Theorem \ref{thm:sfopt}.

\section{Minimum entropy closure}
\label{sec:M1}
For our purposes, solving (\ref{eq:transport}) in full is unnecessary.  We note that both $J_T$ and $J_{SF}$ do not need the exact photon densities as functions of the angular variable; it would be enough to determine the zeroth moment, which is the total energy at a given $(t,x)$:
\begin{equation*}
\psi^{(0)}(t,x)=\int_{S^2}\psi(t,x,\Omega)d\Omega.
\end{equation*}
In light of this, we average (\ref{eq:transport}) over the angular variable and obtain
\begin{equation}
\label{eq:zeromomeqn}
\partial_t\psi^{(0)}(t,x)+\nabla_x\cdot\psi^{(1)}(t,x)+\sigma_t(x)\psi^{(0)}(t,x)=\sigma_s(x)\psi^{(0)}(t,x)+q^{(0)}(t,x)\end{equation}
where 
\begin{equation*}
\psi^{(1)}(t,x)=\int_{S^2}\Omega\psi(t,x,\Omega)d\Omega
\end{equation*}
is the first moment, or the flux vector and $q^{(0)}$ is the average of $q$ over all directions.  By multiplying (\ref{eq:transport}) by $\Omega$ and taking the average over all directions, we get
\begin{equation}
\label{eq:firstmomeqn}
\partial_t\psi^{(1)}(t,x)+\nabla_x\cdot\psi^{(2)}(t,x)+\sigma_t(x)\psi^{(1)}(t,x)=\sigma_s(x)g\psi^{(1)}(t,x)+q^{(1)}(t,x)
\end{equation}
where
\begin{equation*}
\psi^{(2)}(t,x)=\int_{S^2}(\Omega\otimes\Omega)\psi(t,x,\Omega)d\Omega
\end{equation*}
is the second moment, or the pressure tensor of the radiation field, and $q^{(1)}$ is the first moment of the control.  We note that if we continued this process of multiplying (\ref{eq:transport}) by monomials of $\Omega$ and integrating over $S^2,$ we would always have an equation relating the $n^{th}$ moment to the $(n+1)^{th}$ moment.  If we want a closed system, then, we must select an approximation to the $(n+1)^{th}$ moment
\begin{equation*}
\psi^{(n+1)}=D(\psi^{(0)},\dots,\psi^{(n)}).
\end{equation*}
Here, we consider the minimum entropy closure ($M_1$) to approximate $\psi^{(2)}.$ This involves searching for $I$ which minimizes
\begin{equation*}
H^*_R(I)=\int_{S^2}h^*_R(I)d\Omega
\end{equation*}
where
\begin{equation*}
h^*_R(I)=2k\nu^2(n\log n-(n+1)\log(n+1)),\text{ with } n=\frac{I}{2h\nu^3},
\end{equation*}
and $k$ is the Boltzmann constant, $h$ is the Planck constant, and $\nu$ is the frequency of the radiation.  We constrain the entropy minimization by requiring, naturally, that
\begin{equation*}
\int_{S^2} I d\Omega=\psi^{(0)},\text{ and } \int_{S^2} \Omega I d\Omega=\psi^{(1)}.
\end{equation*}
in \cite{DubFeu99}, the entropy minimizer was obtained explicitly by setting 
\begin{equation*}
\psi^{(2)}=D(f)E
\end{equation*}
where the relative flux is $f=\psi^{(1)}/\psi^{(0)},$ the Eddington tensor $D$, is 
\begin{equation*}
D(f)=\frac{1-\chi(f)}{2}I+\frac{3\chi(f)-1}{2}\frac{f\otimes f}{|f|^2}
\end{equation*}
and the Eddington factor $\chi$ is 
\begin{equation*}
\chi(f)=\frac{1}{3}(5-2\sqrt{4-3|f|^2}).
\end{equation*}
This choice of the Eddington factor preserves several reasonable physical properties.  Importantly, the flux is limited; that is, $|f|<1.$ This preserves the finite bound the speed that information travels through the system.  Also, the distribution is nonnegative, which we would like to ensure.  We note that neither property is satisfied for several common closure choices which use a diffusion approximation, such as the spherical harmonics method(also known as $P_N$ methods) \cite{Bru02}.  In \cite{DucDubFra10}, the $M_1$ approximation was seen via numerical tests to provide a high quality approximation when compared with Monte Carlo-based methods.  

Despite these advantages, the $M_1$ does have several drawbacks.  Due to the nonlinearity of the $M_1$ approximation (as opposed to the $P_N$ approximation \cite{FraHerSch08}), we can not expect that the systems obtained by applying the $M_1$ approximation to the optimality systems will match the optimality systems obtained by optimizing the $M_1$ approximation to (\ref{eq:transport}).  Also, as mentioned briefly above, the use of moment models complicates the incorporation of boundary conditions.  The boundary condition (\ref{eq:bdrycond}) describes only the incoming particles, whereas boundary conditions for (\ref{eq:zeromomeqn}) and (\ref{eq:firstmomeqn}) require us to prescribe values for the full moments.  In light of this, we use $U(x)$ to approximate boundary control by setting it as
\begin{equation*}
U(x)=\begin{cases}
q_{max},& \text{if } d_{\partial Z}(x)\leq \epsilon\\
\delta& \text{otherwise},
\end{cases}
\end{equation*}
where $q_{max}>0$ is a total maximum source level, $\epsilon,\delta>0$ are sufficiently small parameters, and $d_{\partial Z}$ denotes the usual distance from the boundary of $Z.$ We note that the first two equations in both (\ref{eq:trackoptsys}) and (\ref{eq:sfoptsys}) are very similar and we thus use the $M_1$ closure to solve the state and adjoint equations.  That is, we take an optimize-then-discretize approach to $\mathcal{P}_T$ and $\mathcal{P}_{SF}$ as opposed to a discretize-then-optimize approach.

\section{2-D Numerical Tests}
\label{sec:numtest}
In this section, we study the 2-dimensional region $Z:=[-1,1]\times[-1,1].$ We define a void-like region $Z_{V}:=\{0.8\leq|x|\leq0.9\}$ On $Z_V,$ the scattering  and absorbtion cross-sections $\sigma_s,\sigma_a$ have values that differ from those elsewhere; these are summarized in Table \ref{sigmatable} for the specific values.

\begin{table}
\centering
\begin{tabular}{| c | c | c |} 
\hline
 &  $Z_V$ & $Z\setminus Z_V$\\ \hline
 $\sigma_a$ & 0.001 & 0.05 \\ \hline
 $\sigma_s$ & 0.01 & 0.5 \\
\hline
\end{tabular} 
\caption{Material parameters for $\mathcal{P}_{SF}$}
\label{sigmatable}
\end{table}
\begin{table}
\centering
\begin{tabular}{| c | c | c | c |} 
\hline
 & $Z_T$ & $Z_R$\ & $Z_N$\\ \hline
 $\alpha$ & 0.52 & 0.170& 0.170 \\ \hline
 $\beta$ & 0.171 & 0.0078& 0.0078 \\
\hline
\end{tabular} 
\caption{Cell-response parameters for $\mathcal{P}_{SF}$}
\label{lqtable}
\end{table}
The final parameter of the physical medium to be defined is the scattering phase function.  For simplicity in our test problem, we define $g\equiv 0.85.$  Other, more involved, kernels might be used in this range for $g$ to agree with Mie scattering theory as in \cite{AydOliGod02}.

For testing, we use three tumor/risk region regions similar to those in \cite{SheFerOli99}.  Specifically, we define the regions in Table \ref{tumortable} and are shown in Figure \ref{fig:cases}; the void region is shown in black and the tumor and risk regions are traced in white.  In the basic target case, seen in Figure \ref{fig:basiccase}, the tumor region is a box, as is the risk region.  The second, intermediate target case, seen in Figure \ref{fig:intercase}, involves an L-shaped tumor around a box-shaped risk region.  Finally, the complex target case in Figure \ref{fig:complexcase} involves a C-shaped tumor around a risk region.
\begin{table}
\centering
\begin{tabular}{| c | c | c |} 
\hline
  &  $Z_T$ & $Z_R$ \\ \hline
Basic Target & $[-0.25,0.25]\times[-0.25,0.25]$ & $[.254, 0.379]\times[-0.125,0.125]$ \\ \hline 
\multirow{2}{*}{Intermediate Target} & $([-0.25,0.25]\times[-0.25,0])$ & $[0.04,0.25]\times[0.04,0.25]$\\ 

 &$\cup([-0.25,0]\times[0,0.25])$ &  \\ \hline
\multirow{3}{*}{Complex Target}  & $([-0.25,0.25]\times[-0.25,-0.125])$ & $[-0.25,-0.04]\times[-0.121,0.121]$\\
 & $\cup([-0.25,0.25]\times[0.125,0.25])$ & \\
 & $\cup([0.04,0.25]\times[-0.125,0.125])$ &  \\
\hline
\end{tabular} 
\caption{Locations of Tumor and Risk Region}
\label{tumortable}
\end{table}
  We will solve both $\mathcal{P}_T$ and $\mathcal{P}_{SF}$ for each geometry seen in Figure \ref{fig:cases}.  For $\mathcal{P}_T$ for each example, we set $\overline{D}\equiv T$ on $Z_T$ and $0$ elsewhere, corresponding to an average (over time) dose of $1,$ and 
\begin{equation*}
c_1=c_T\chi_{Z_T}+c_R\chi_{Z_R}+c_N\chi_{Z_N}
\end{equation*}
 with $c_T=25,~c_R=150,~c_N=1.$ For $\mathcal{P}_{SF},$ we set $\rho_0\equiv\chi_{Z_T},~\rho_1\equiv\chi_{Z_R},~\rho_2\equiv\chi_{Z_N},$ with weights $a_0=500,~a_1=2000,~a_2=1.$  The LQ parameters are shown below; the tumor region uses parameters for Lewis Lung tumor cells and the risk region and normal region use parameters for V79 (normal Chinese hamster lung tissue); both sets of parameters are taken from \cite{SteDowPea86}.
\begin{figure}
\centering
\subfigure[Basic Target]{
\begin{tikzpicture}[scale=2.5]
\fill[gray!75] (-1,-1)--(-1,1)--(1,1)--(1,-1)--(-1,-1);
\fill[black, even odd rule] (0,0) circle (.8) (0,0) circle (.9);
\draw[white,line width=0.2pt] (-0.25,-0.25)--(-0.25,0.25)--(0.25,0.25)--(0.25,-0.25)--cycle;
\draw[white,line width=0.2pt] (.29,-0.125)--(0.29,0.125)--(0.379,0.125)--(0.379,-0.125)--cycle;
\end{tikzpicture}
\label{fig:basiccase}
}
\subfigure[Intermediate Target]{
\begin{tikzpicture}[scale=2.5]
\fill[gray!75] (-1,-1)--(-1,1)--(1,1)--(1,-1)--(-1,-1);
\fill[black, even odd rule] (0,0) circle (.8) (0,0) circle (.9);
\draw[white,line width=0.2pt] (-0.25,-0.25)--(0.25,-0.25)--(0.25,0)--(0,0)--(0,0.25)--(-0.25,0.25)--cycle;
\draw[white,line width=0.2pt] (0.06,0.06)--(0.25,0.06)--(0.25,0.25)--(0.06,0.25)--cycle;
\end{tikzpicture}
\label{fig:intercase}
}
\subfigure[Complex Target]{
\begin{tikzpicture}[scale=2.5]
\fill[gray!75] (-1,-1)--(-1,1)--(1,1)--(1,-1)--(-1,-1);
\fill[black, even odd rule] (0,0) circle (.8) (0,0) circle (.9);
\draw[white,line width=0.2pt] (-0.25,-0.25)--(0.25,-0.25)--(0.25,0.25)--(-0.25,0.25)--(-0.25,0.125)--(0.04,0.125)--(0.04,-0.125)--(-0.25,-0.125)--cycle;
\draw[white,line width=0.2pt] (-0.25,-0.09)--(-0.04,-0.09)--(-0.04,0.09)--(-0.25,0.09)--cycle;
\end{tikzpicture}

\label{fig:complexcase}
}
\caption{Example cases}
\label{fig:cases}
\end{figure}
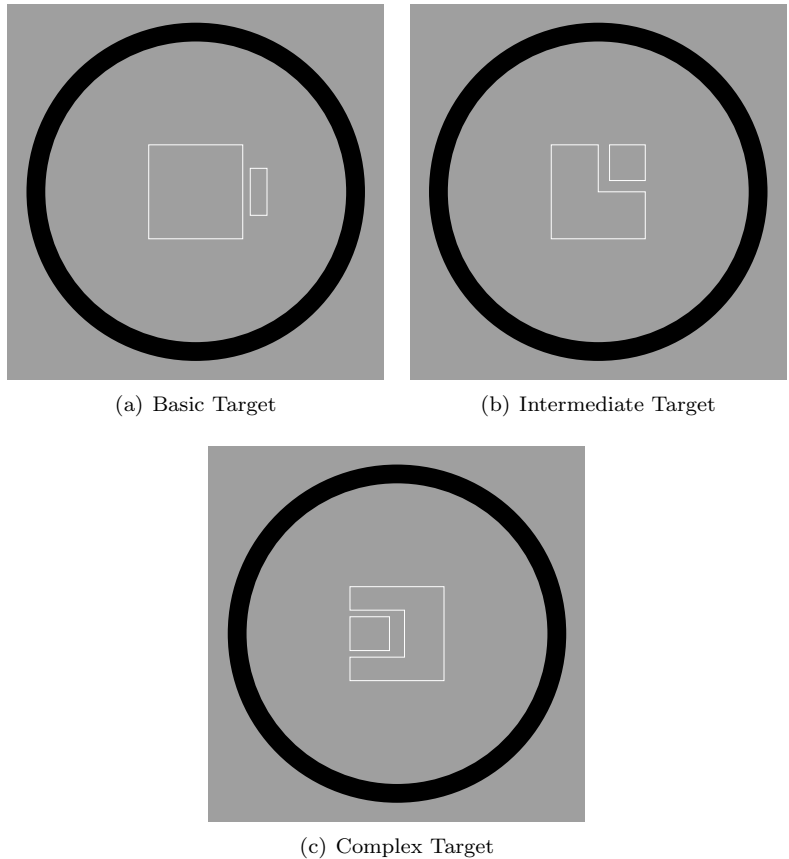
\setcounter{subfigure}{0}
We use a simple gradient descent algorithm to determine an optimal control for each of the three targets and for each objective functional.  The $M_1$ approximations to both the state and adjoint equations are solved with a finite volume solver using approximately 10,000 cells with final time $T=5.$ The parameters $\delta$ and $\epsilon$ in the definition of $U(x)$ are chosen as $\min\{\Delta x,\Delta y\}(10^{-4})t$ and $\min\{\Delta x,\Delta y\},$ respectively, where $\Delta x,\Delta y$ are the mesh sizes.  We consider the algorithm as converged if both the difference between consecutive iterates and $||j'(q)||_{\infty}$ are smaller than $10^{-4}.$  

\begin{figure}
\centering
\subfigure[Tracking, Basic Target]{
\label{fig:trackbasic}
\includegraphics [height=5cm] {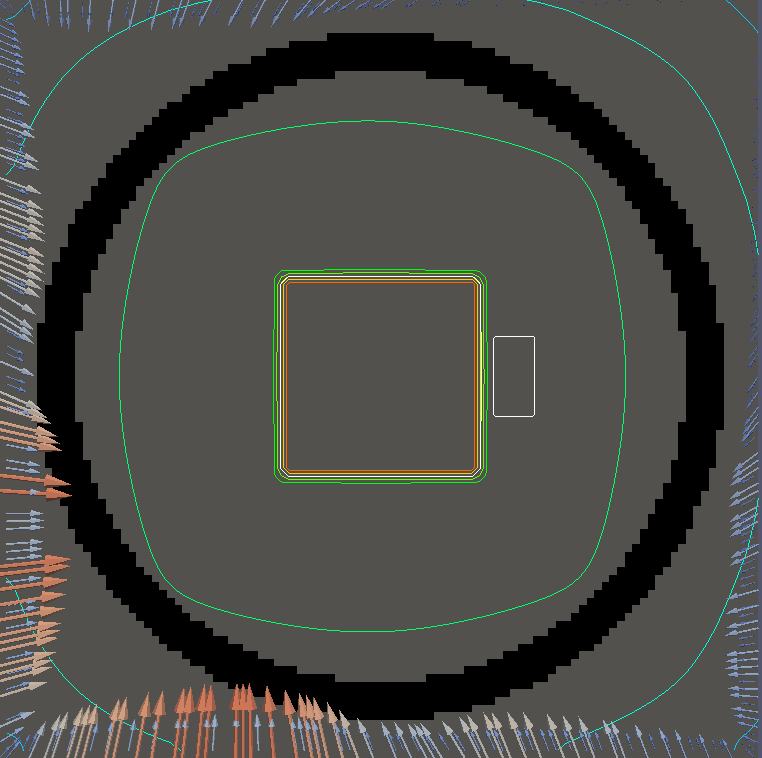}
}
\subfigure[Surviving Fraction, Basic Target]{
\label{fig:sfbasic}
\includegraphics [height=5cm] {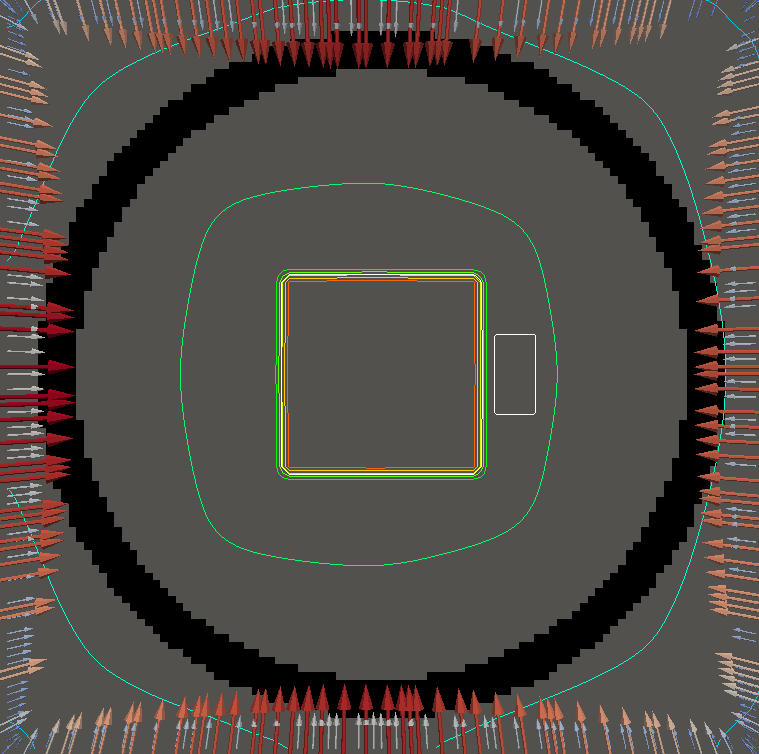}
}\\
\subfigure[Tracking, Intermediate Target]{
\label{fig:trackinter}
\includegraphics [height=5cm] {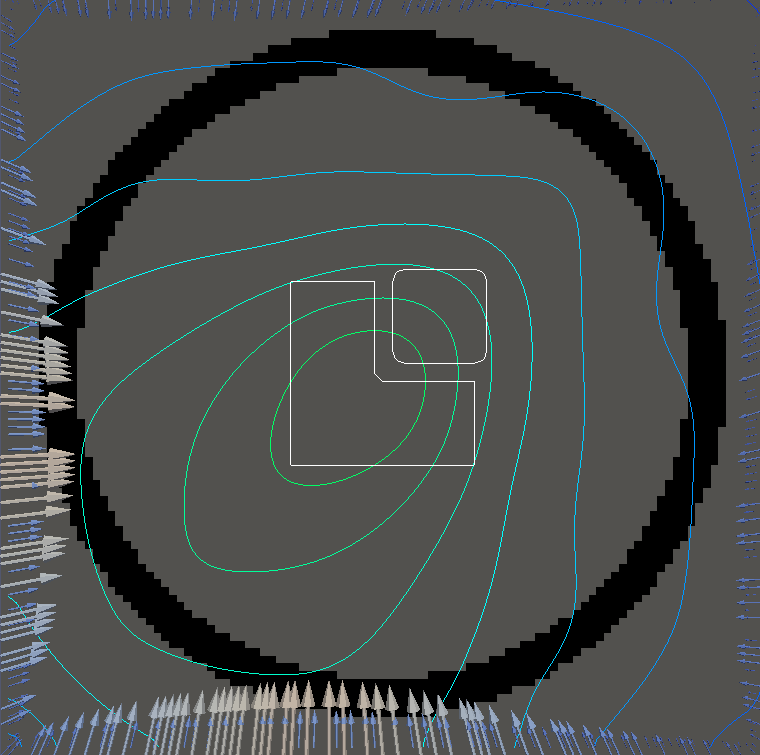}
}
\subfigure[Surviving Fraction, Intermediate Target]{
\label{fig:sfinter}
\includegraphics [height=5cm] {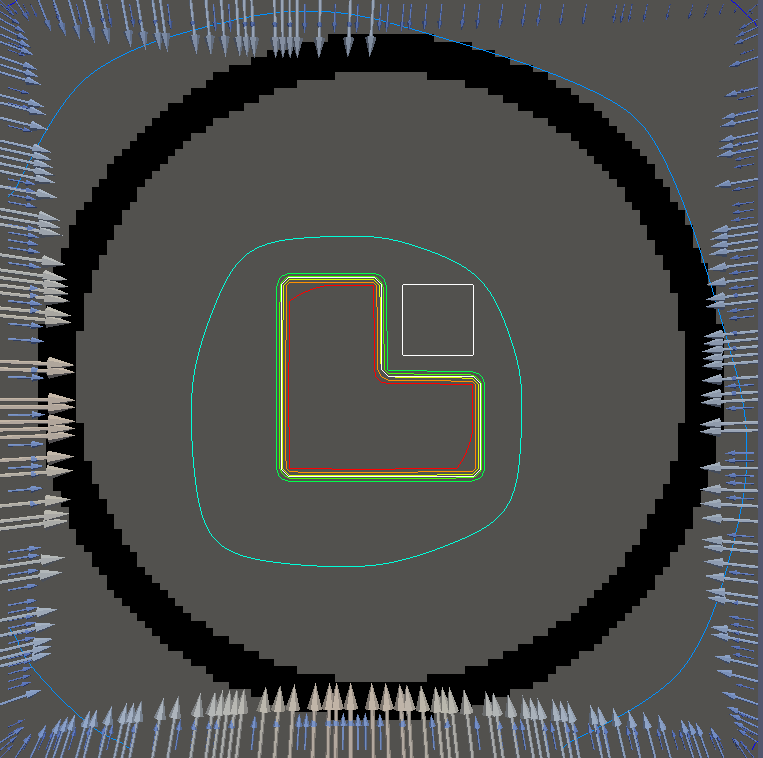}
}\\
\subfigure[Tracking, Complex Target]{
\label{fig:trackcomplex}
\includegraphics[height=5cm]  {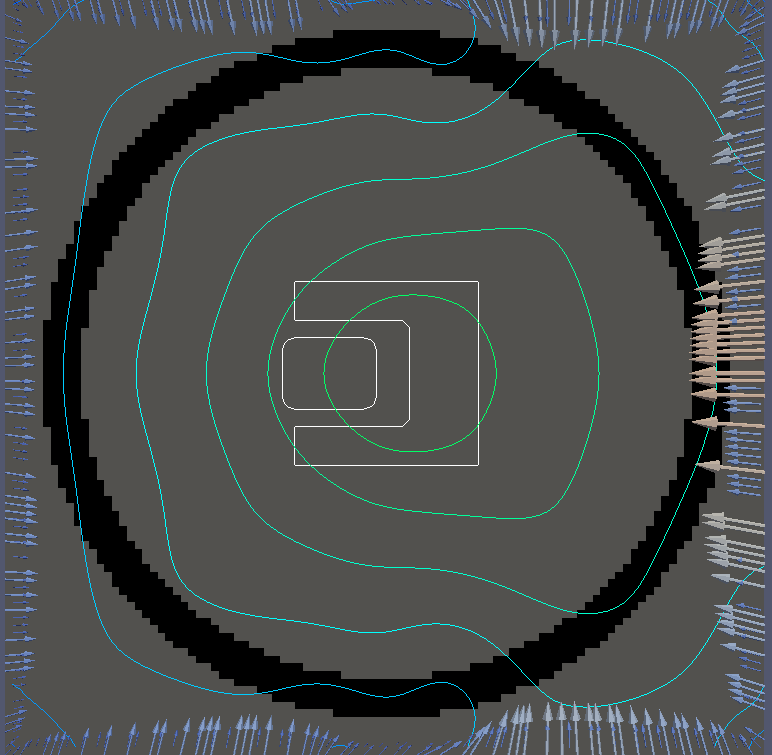}
}
\subfigure[Surviving Fraction, Complex Target]{
\label{fig:sfcomplex}
\includegraphics [height=5cm] {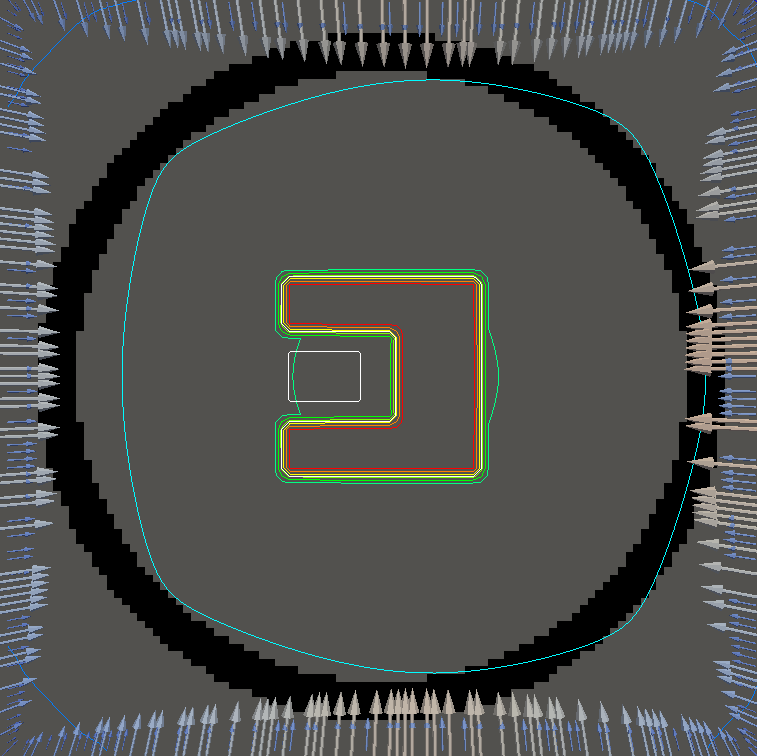}
}
\caption{Optimal boundary source for $\mathcal{P}_T,$ in \subref{fig:trackbasic},\subref{fig:trackinter},\subref{fig:trackcomplex} and $\mathcal{P}_{SF}$ in \subref{fig:sfbasic},\subref{fig:sfinter},\subref{fig:sfcomplex}}
\label{fig:optdose}
\end{figure}
\setcounter{subfigure}{0}

Figure \ref{fig:optdose} shows the optimal boundary source term for both $\mathcal{P}_T$ and $\mathcal{P}_{SF}.$  The vectors shown on the boundary are the time-integrated values of $q^{(1)}$ normalized and then scaled by $q^{(0)}.$  In Figures \ref{fig:trackbasic}, \ref{fig:trackinter}, and \ref{fig:trackcomplex} (corresponding to $\mathcal{P}_T$), the isolines are spaced at $5\%$ intervals of the maximum of the desired dose (here, 5).  In the intermediate and tracking cases, we see that relatively low dose levels are attained, primarily due to the high penalty to any dose deposited in the risk region.  In Figures \ref{fig:sfbasic}, \ref{fig:sfinter}, and \ref{fig:sfcomplex}(corresponding to $\mathcal{P}_{SF}$), the isolines are spaced at intervals of $10\%$ of cells killed.  Here a high proportion of the tumor cells are killed (in each case $80\%-90\%$) while in the Intermediate and Basic cases, the tumor has at least $60\%$ survival; in the Complex case, the risk region has $50\%-60\%$ survival.  

The dose deposited in $\mathcal{P}_T$ changes significantly when we alter the relative weights $c_T$ and $c_R$. In Figure \ref{fig:optlowrisk}, we see the results for solving $\mathcal{P}_T,$ and $\mathcal{P}_{SF}$ where we set $c_R=50$ and $a_1=1000$ (all other parameters are unchanged).  In both the basic and intermediate cases, the dose delivered to the tumor is significantly higher while also remaining largely concentrated on the tumor.  This is slightly less true in the complex case.  In the figures for $\mathcal{P}_{SF},$ there is no general change in the pattern of cell death, however the risk region in the first two cases has $50\%$ cell survival whereas the complex case has approximately $40\%-50\%$ cell survival.  However, this lack of change in pattern can partially be attributed to the tumor cells being more susceptible to the radiation dose.  

\begin{figure}
\centering
\subfigure[Tracking, Basic Target]{
\label{fig:trackbasiclow}
\includegraphics[height=5cm] {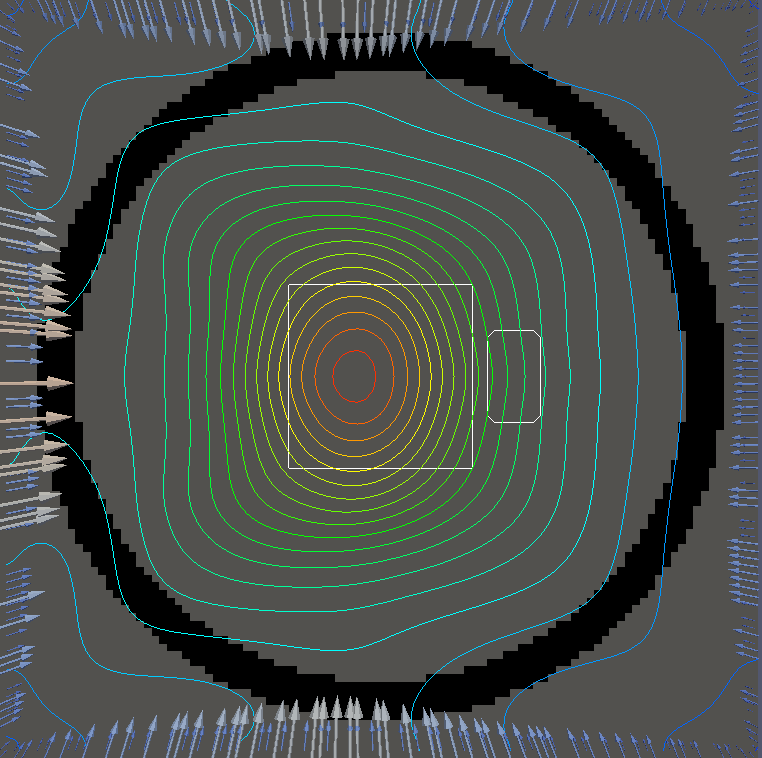}
}
\subfigure[Surviving Fraction, Basic Target]{
\label{fig:sfbasiclow}
\includegraphics[height=5cm] {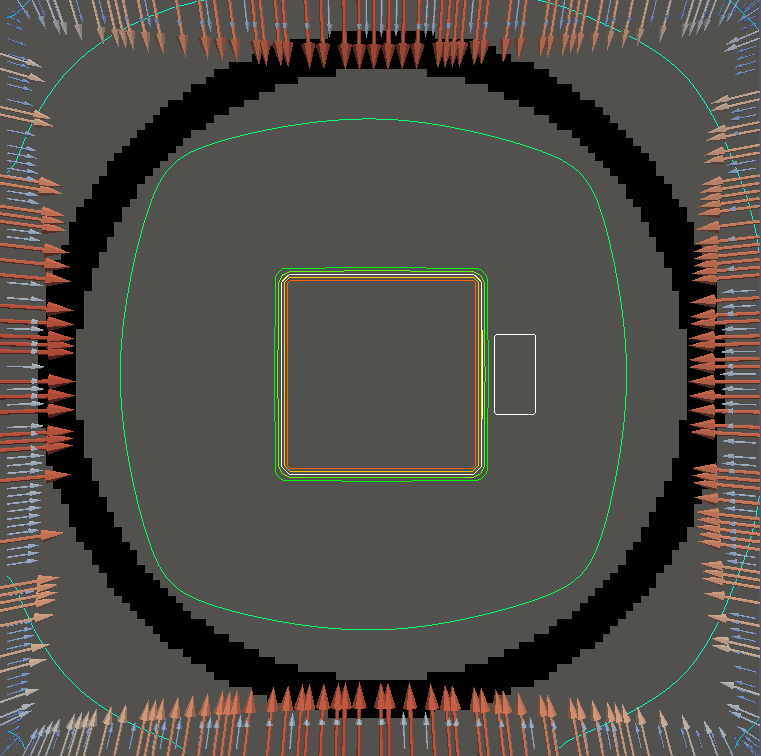}
}\\
\subfigure[Tracking, Intermediate Target]{
\label{fig:trackinterlow}
\includegraphics[height=5cm]{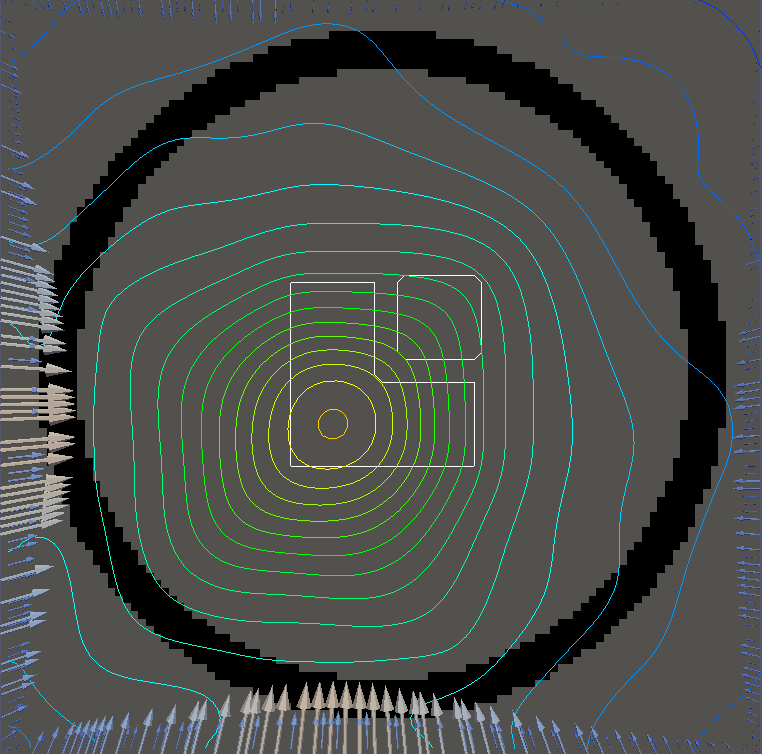}
}
\subfigure[Surviving Fraction, Intermediate Target]{
\label{fig:sfinterlow}
\includegraphics[height=5cm]{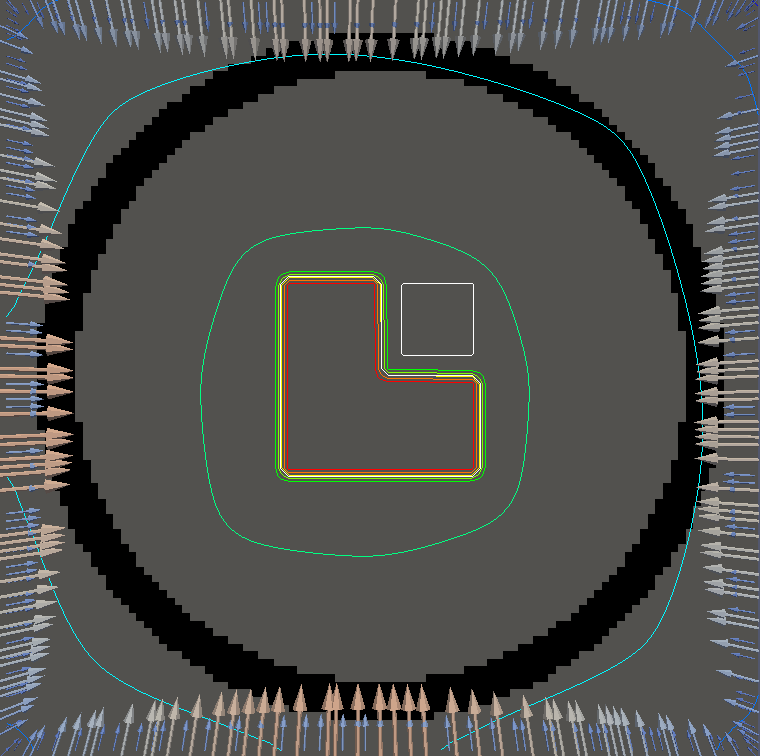}
}\\
\subfigure[Tracking, Complex Target]{
\label{fig:trackcomplexlow}
\includegraphics [height=5cm]{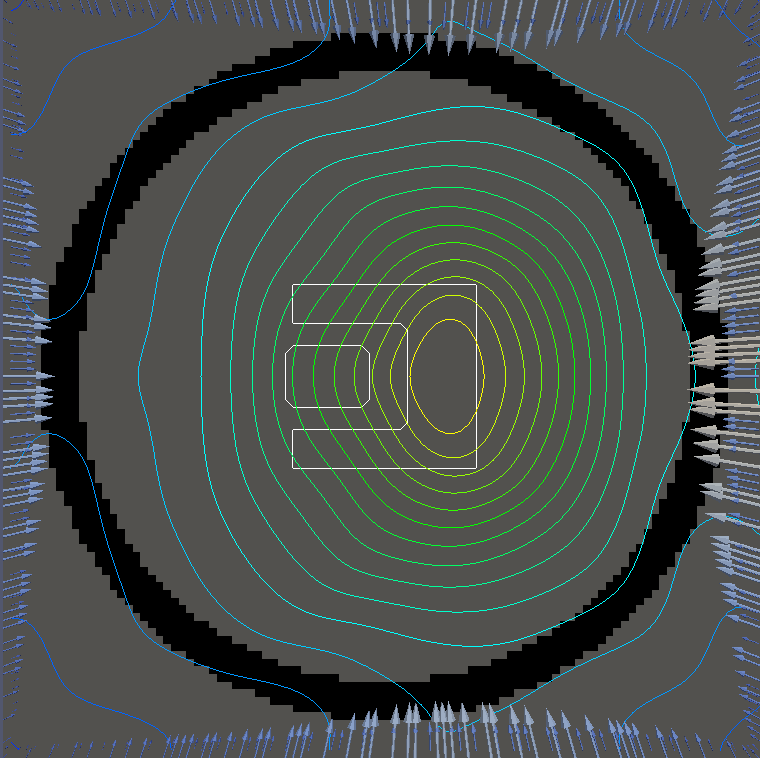}
}
\subfigure[Surviving Fraction, Complex Target]{
\label{fig:sfcomplexlow}
\includegraphics[height=5cm] {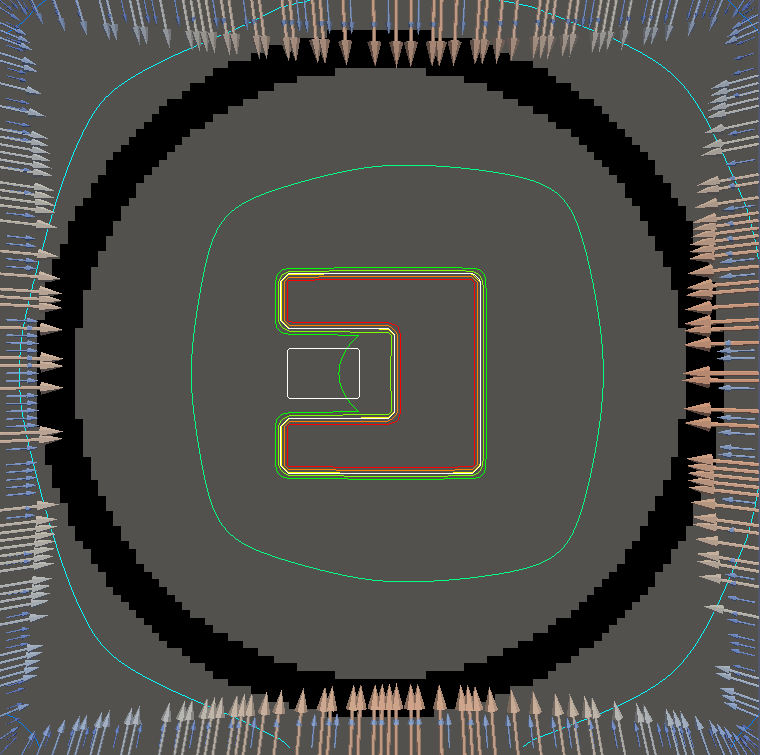}
}
\caption{Optimal results with lower penalty to dose in $Z_R$}
\label{fig:optlowrisklow}
\end{figure}
\setcounter{subfigure}{0}

We conclude with a final set of numerical examples which restrict the location of the source by altering the definition of $U(x).$  Here we require that $q\leq \delta$ on one side of the boundary.  For the basic and intermediate case, we require that the external source not come from the left side of $Z.$  For the complex case, we disallow sources on the right side  (as the optimal source is nearly zero on the right side in the complex case for $\mathcal{P}_T$).  Figure \ref{fig:optblocked} shows the optimal solution for both problems, using the same penalization parameters used in Figure \ref{fig:optdose}.  The optimal dose for $\mathcal{P}_{T}$ is significantly worse, with the tumor in the intermediate and complex cases getting a dose below $3.$ However, the tumor cells have a survival of $10\%$ or less for each case and the risk region has a survival rate of $60\%$ or higher in each case. 
\begin{figure}
\centering
\subfigure[Tracking, Basic Target]{
\label{fig:trackbasicblock}
\includegraphics [height=5cm]{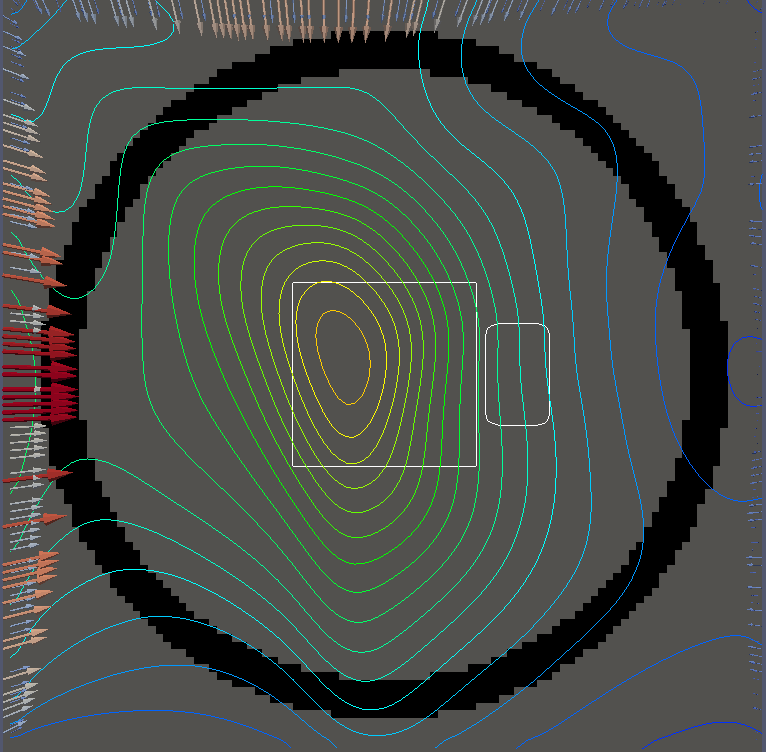}
}
\subfigure[Surviving Fraction, Basic Target]{
\label{fig:sfbasicblock}
\includegraphics [height=5cm]{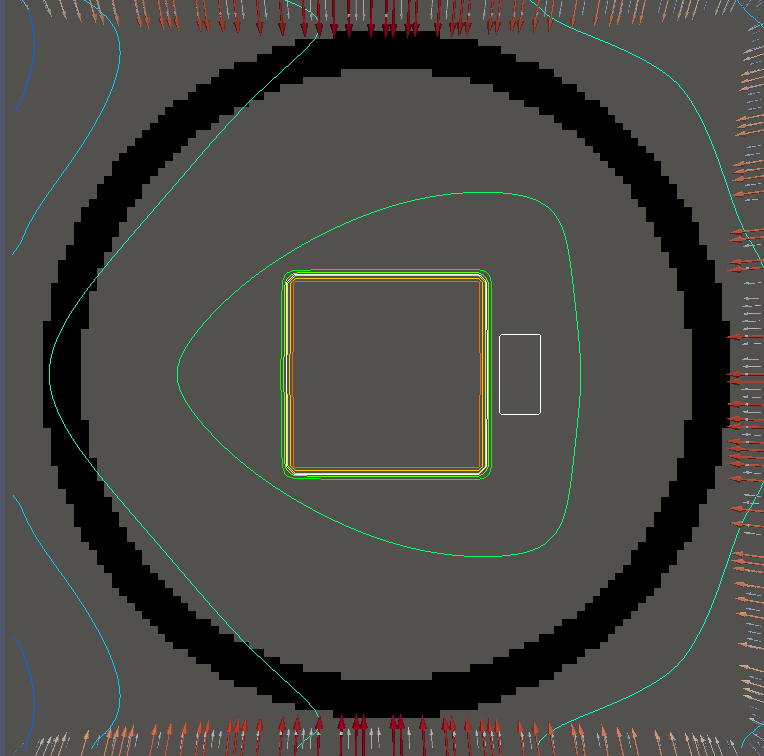}
}\\
\subfigure[Tracking, Intermediate Target]{
\label{fig:trackinterblock}
\includegraphics [height=5cm]{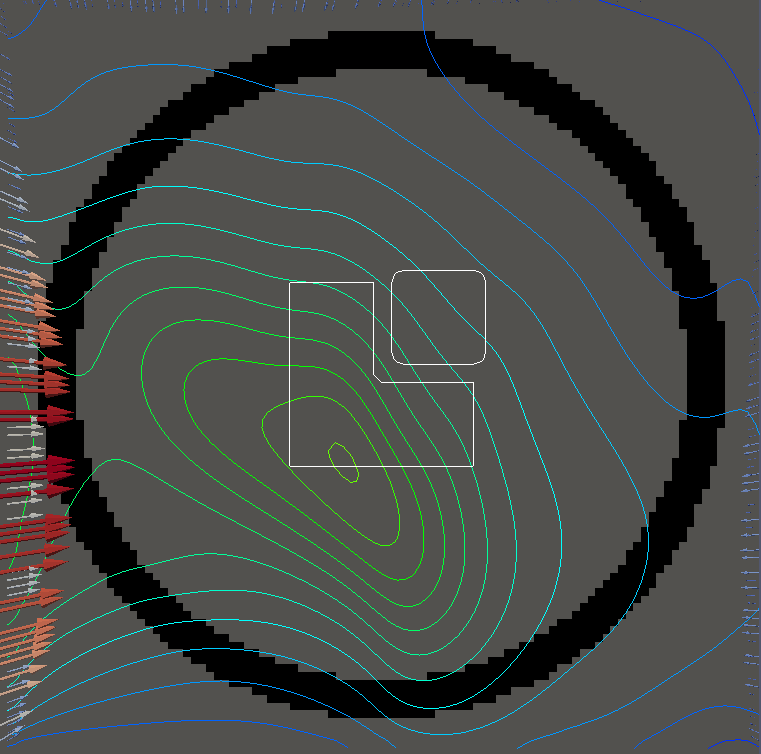}
}
\subfigure[Surviving Fraction, Intermediate Target]{
\label{fig:sfinterblock}
\includegraphics[height=5cm]{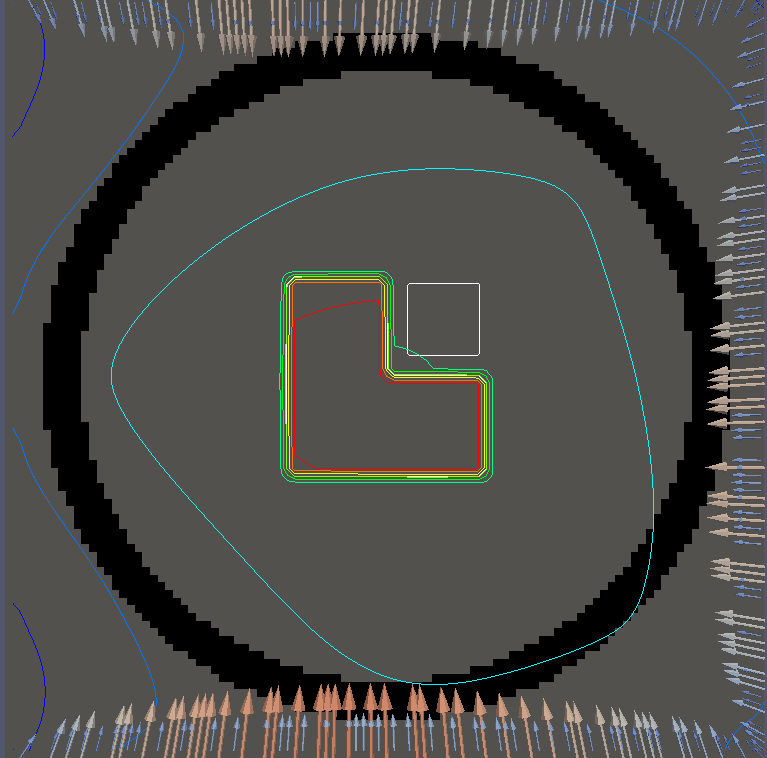}
}\\
\subfigure[Tracking, Complex Target]{
\label{fig:trackcomplexblock}
\includegraphics [height=5cm]{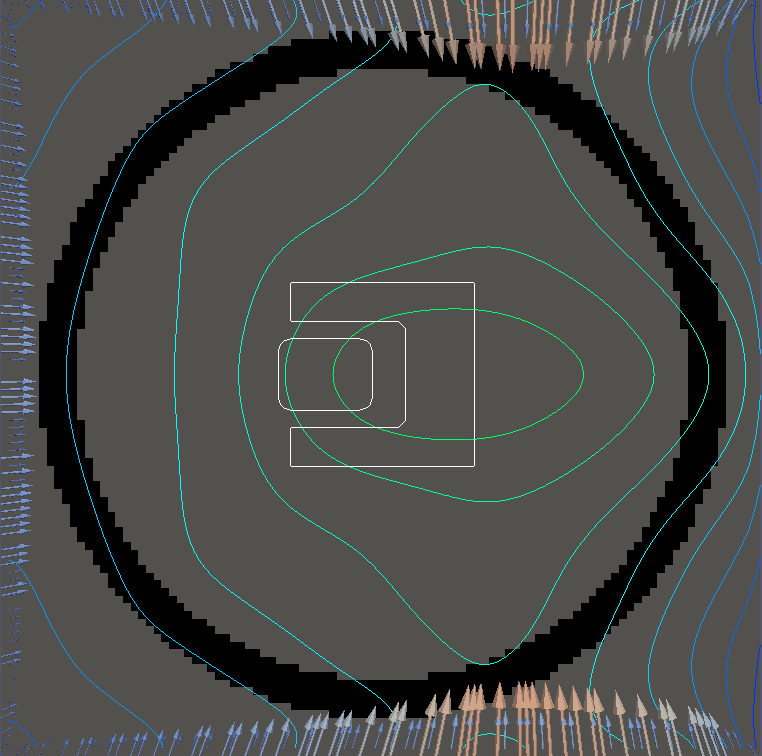}
}
\subfigure[Surviving Fraction, Complex Target]{
\label{fig:sfcomplexblock}
\includegraphics [height=5cm]{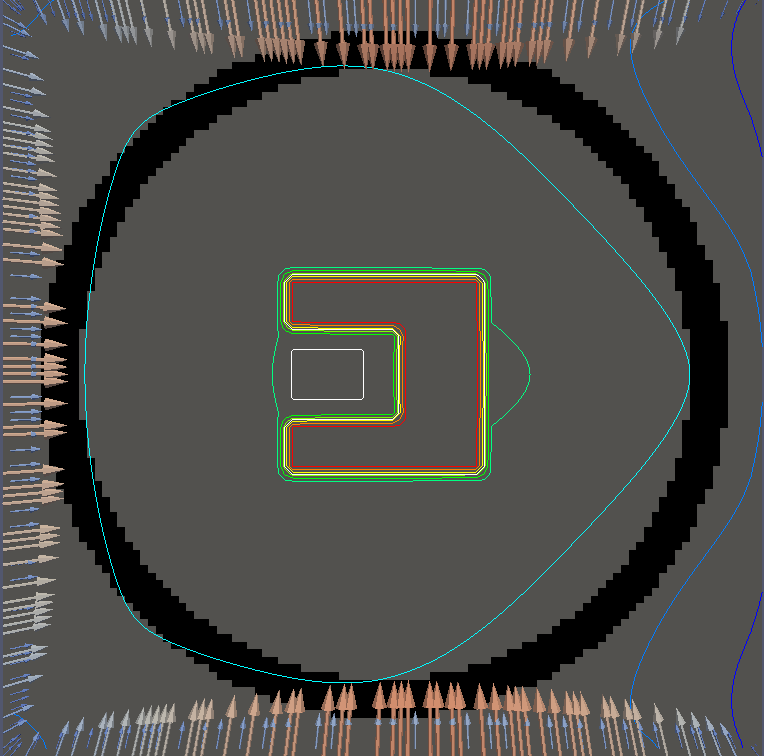}
}
\caption{Optimal results with one edge blocked}
\label{fig:optblocked}
\end{figure}
\setcounter{subfigure}{0}

Clearly, the choice of relative weights $c_T,c_R$ and $a_i$ affects the optimal solution.  However, we have currently no systematic {\em a priori} method of selecting the relative weights in either problem.  Such a method should take into account the relative importance of the different tissues and the susceptibilities of the various cell types, leading to a possibly nonintuitive scheme.  Without such a method, several optimization runs may be required to find an appropriate treatment strategy.  We also note that the optimization does not consider the feasibility of the optimal source. In fact, the optimal boundary control in all of the numerical results are spread out along the boundary, as opposed to being focused in a beam configuration.  Further modeling would be required in order for the cost function to penalize non-beam like configurations.

\section*{Acknowledgments} 
This work has been supported by HE5386/8-1, FR 2841/1-1 and the Seedfunds of RWTH Aachen. 

\bibliographystyle{gOPT}
\bibliography{references}

\begin{thebibliography}{10}

\bibitem{AydOliGod02}
{\sc E.~D. Aydin, C.~R.~E. de~Oliveira, and A.~J.~H. Goddard}, {\em A
  comparison between transport and diffusion calculations using a finite
  element-spherical harmonics radiation transport method}, Medical Physics, 29
  (2002), pp.~2013--2023.

\bibitem{Bor99}
{\sc C.~B{\"o}rgers}, {\em The radiation therapy planning problem}, in
  Computational radiology and imaging ({M}inneapolis, {MN}, 1997), vol.~110 of
  IMA Vol. Math. Appl., Springer, New York, 1999, pp.~1--16.

\bibitem{Bru02}
{\sc T.~Brunner}, {\em Forms of approximate radiation transport}, Sandia
  Report,  (2002).

\bibitem{Dal85}
{\sc R.~G. Dale}, {\em The application of the linear-quadratic dose-effect
  equation to fractionated and protracted radiotherapy}, Br J Radiol, 58
  (1985), pp.~515--528.

\bibitem{DauLio93}
{\sc R.~Dautray and J.-L. Lions}, {\em Mathematical analysis and numerical
  methods for science and technology.}, vol.~6, Springer-Verlag, Berlin, 1993.

\bibitem{DubFeu99}
{\sc B.~Dubroca and J.-L. Feugeas}, {\em \'{E}tude th\'{e}orique et
  num\'{e}rique d'une hi\'{e}rarchie de mod\`{e}les aux moments pout le
  transfert radiatif}, Analyse num\'{e}rique, 329 (1999), pp.~915--920.

\bibitem{DucDubFra10}
{\sc R.~Duclous, B.~Dubroca, and M.~Frank}, {\em A deterministic partial
  differential equation model for dose calculation in electron radiotherapy},
  Physics in Medicine and Biology, 55 (2010), p.~3843.

\bibitem{FerMavAda08}
{\sc B.~C. Ferreira, P.~Mavroidis, M.~Adamus-G\'{o}rka, R.~Svensson, and B.~K.
  Lind}, {\em The impact of different dose-response parameters on biologically
  optimized imrt in breast cancer}, Physics in Medicine and Biology, 53 (2008),
  p.~2733.

\bibitem{FraHerSan10}
{\sc M.~Frank, M.~Herty, and A.~N. Sandjo}, {\em Optimal radiotherapy treatment
  planning governed by kinetic equations}, Math. Models Methods Appl. Sci., 20
  (2010), pp.~661--678.

\bibitem{FraHerSch08}
{\sc M.~Frank, M.~Herty, and M.~Sch{\"a}fer}, {\em Optimal treatment planning
  in radiotherapy based on {B}oltzmann transport calculations}, Math. Models
  Methods Appl. Sci., 18 (2008), pp.~573--592.

\bibitem{KinDiPWaz00}
{\sc C.~R. King, T.~A. DiPetrillo, and D.~E. Wazer}, {\em Optimal radiotherapy
  for prostate cancer: predictions for conventional external beam, imrt, and
  brachytherapy from radiobiologic models}, International Journal of Radiation
  Oncology*Biology*Physics, 46 (2000), pp.~165 -- 172.

\bibitem{KufMonSch09}
{\sc K.-H. K{\"u}fer, M.~Monz, A.~Scherrer, P.~S{\"u}ss, F.~Alonso, A.~S.~A.
  Sultan, T.~Bortfeld, and C.~Thieke}, {\em Multicriteria optimizaton in
  intensity modulated radiotherapy planning}, in Handbook of optimization in
  medicine, vol.~26 of Springer Optim. Appl., Springer, New York, 2009,
  pp.~123--167.

\bibitem{SheFerOli99}
{\sc D.~M. Shepard, M.~C. Ferris, G.~H. Olivera, and T.~R. Mackie}, {\em
  Optimizing the delivery of radiation therapy to cancer patients}, SIAM
  Review, 41 (1999), pp.~721--744.

\bibitem{SouParEva08}
{\sc C.~P. South, M.~Partridge, and P.~M. Evans}, {\em A theoretical framework
  for prescribing radiotherapy dose distributions using patient-specific
  biological information}, Medical Physics, 35 (2008), pp.~4599--4611.

\bibitem{SteDeaDuc87}
{\sc G.~G. Steel, J.~M. Deacon, G.~M. Duchesne, A.~Horwich, L.~R. Kelland, and
  J.~H. Peacock}, {\em The dose-rate effect in human tumour cells},
  Radiotherapy and Oncology, 9 (1987), pp.~299 -- 310.

\bibitem{SteDowPea86}
{\sc G.~G. Steel, J.~D. Down, J.~H. Peacock, and T.~C. Stephens}, {\em
  Dose-rate effects and the repair of radiation damage}, Radiotherapy and
  Oncology, 5 (1986), pp.~321 -- 331.

\bibitem{Tro10}
{\sc F.~Tr{\"o}ltzsch}, {\em Optimal control of partial differential equations:
  theory, methods, and applications}, vol.~v. 112 of Graduate studies in
  mathematics, American Mathematical Society, Providence, R.I., 2010.

\bibitem{Web94}
{\sc S.~Webb}, {\em Optimum parameters in a model for tumour control
  probability including interpatient heterogeneity}, Physics in Medicine and
  Biology, 39 (1994), pp.~1895--1914.

\end{thebibliography}
\end{document}